\documentclass{amsart}

\usepackage[usenames]{xcolor}
\usepackage{graphicx}
\usepackage{amssymb}
\usepackage{latexsym}
\usepackage{amsmath}
\usepackage{enumerate}
\usepackage{amsthm}

\newtheorem{theorem}{Theorem}[section]
\newtheorem{define}[theorem]{Definition}

\newtheorem{exa}[theorem]{Example}
\newenvironment{example}{\begin{exa} \rm}{\qee\end{exa}}
\newtheorem{propo}[theorem]{Proposition}

\newtheorem{exerc}[theorem]{Exercise}

\newtheorem{conj}[theorem]{Conjecture}

\newtheorem{ques}[theorem]{Open Question}
\newenvironment{question}{\begin{ques} \rm}{\qee\end{ques}}
\newtheorem{lem}[theorem]{Lemma}
\newenvironment{lemma}{\begin{lem} \it}{\end{lem}}
\newtheorem{cor}[theorem]{Corollary}
\newenvironment{corollary}{\begin{cor} \it}{\end{cor}}
\newtheorem{factief}[theorem]{Fact}
\newenvironment{fact}{\begin{factief} \it}{\end{factief}}
\newtheorem{rem}[theorem]{Remark}
\newenvironment{remark}{\begin{rem} \rm}{\qee\end{rem}}
\newtheorem{dis}[theorem]{Discussion}

\renewcommand{\iff}{\leftrightarrow}
\newcommand{\medent}{\medskip\noindent}
 \newcommand{\tupel}[1]{{\langle #1 \rangle}}
\newcommand{\verz}[1]{\{ #1 \}}

\newcommand{\shapipe}{\,
                    \setlength{\unitlength}{1ex}
                    \begin{picture}(2,2)
                    \put(.25,.15){\line(0,1){1.22}}
                    \put(.25,1.37){\line(1,0){0.61}}
                    \put(.45,-.05){\line(1,0){0.61}}
                    \put(.4,0){\line(1,0){0.61}}
                    \put(.35,.05){\line(1,0){0.61}}
                    \put(.3,.1){\line(1,0){0.61}}
                    \put(.25,.15){\line(1,0){0.61}}
                    \put(1.04,-.05){\line(0,1){1.22}}
                    \put(0.99,0){\line(0,1){1.22}}
                    \put(0.94,.05){\line(0,1){1.22}}
                    \put(0.89,.1){\line(0,1){1.22}}
                    \put(0.84,.15){\line(0,1){1.22}}
                    \end{picture}
                    \!}

\newcommand{\qee} {\hspace*{2mm}\hfill $\shapipe$}
\newcommand{\sat}{{\sf sat}}
\newcommand{\apr}{{\vartriangle}}
\newcommand{\opr}{{\Box}}

 \newcommand{\nrhd}{\mathrel{\not\! \rhd}}

\newcommand{\jump}{\mathrel{\mbox{\textcolor{gray}{$\blacktriangleright$}}}}
\newcommand{\jumpb}{\mathrel{\mbox{\textcolor{gray}{$\blacktriangleleft$}}}}
\newcommand{\sline}{\raise-0.3ex\hbox{$\hbox{--}\kern-0.84ex\raise0.45ex\hbox{$\hbox{\scalebox{0.3}{\bf /}}\kern-0.37ex\hbox{\scalebox{0.3}{\bf /}}$}$}}
\newcommand{\slinei}{\raise-0.3ex\hbox{$\hbox{--}\kern-0.84ex\raise0.45ex\hbox{$\hbox{\scalebox{0.3}{\bf \textbackslash}}\kern-0.37ex\hbox{\scalebox{0.3}{\bf \textbackslash}}$}$}}
\newcommand{\jumpneq}{\mathrel{\jump_{\hspace*{-0.235cm}{}_{\kern0.08ex \sline}}\hspace*{0.09cm}}}
\newcommand{\jumpbneq}{\mathrel{\jumpb_{\hspace*{-0.235cm}{}_{\kern0.08ex \slinei}}\hspace*{0.09cm}}}
\newcommand{\dashvneq}{\mathrel{\dashv_{\hspace*{-0.235cm}{}_{\kern0.2ex \slinei}}\hspace*{0.09cm}}}
\newcommand{\vdashneq}{\mathrel{\vdash_{\hspace*{-0.235cm}{}_{\kern0.3ex \sline}}\hspace*{0.09cm}}}
\newcommand{\lhdnneq}{\mathrel{\lhd_{\hspace*{-0.27cm}{}_{\kern0.2ex \slinei}}\hspace*{0.09cm}}}
\newcommand{\rhdnneq}{\mathrel{\rhd_{\hspace*{-0.27cm}{}_{\kern0.3ex \sline}}\hspace*{0.09cm}}}
  
 \newcommand{\noshow}[1]{}
 \newcommand{\opcl}[1]{(\noshow{)}#1\noshow{[}]}
  \newcommand{\valf}[1]{{\sf F}_{{#1}}}
\newcommand{\tinyset}{{\sf TiS}}

\title[On a Question of Krajewski's]{On a Question of Krajewski's}

\author{Fedor Pakhomov}
 \address{Steklov Mathematical Institute,
                8 Gubkina Street,
                Moscow 119991, Russia}
\email{pakhfn@mi.ras.ru}

\author{Albert Visser}
 \address{Philosophy, Faculty of Humanities,
                Utrecht University,
               Janskerkhof 13,
                3512BL~~Utrecht, The Netherlands}
\email{a.visser@uu.nl}
\date{\today}



\thanks{We are grateful to Leszek Ko{\l}odziejczyk for an enlightening discussion of the question answered
in Section~\ref{unapred}.}

\begin{document}

\begin{abstract}
In this paper we provide a (negative) solution to a problem posed by Stanis{\l}aw Krajewski. 
Consider a  recursively enumerable theory $U$ and a finite expansion of the signature of $U$
that contains at least one predicate symbol of arity $\geq 2$. 
We show that, for any finite extension $\alpha$ of $U$ in the expanded language
that is conservative over $U$, there is a conservative extension $\beta$ of $U$ in the expanded language, such that $  \alpha\vdash\beta$
and $\beta \nvdash \alpha$. The result is preserved when we consider either \emph{extensions} or \emph{model-conservative extensions}
of $U$ in stead of \emph{conservative extensions}. Moreover, the result is preserved 
 when we replace $\dashv$ as ordering on the finitely axiomatized extensions in the expanded language
by a special kind of interpretability, to wit \emph{interpretability that identically translates the symbols of the
$U$-language}.
 
 We show that the result fails when we consider an expansion with only unary predicate symbols for conservative extensions of $U$
 ordered by interpretability that preserves the symbols of $U$.
\end{abstract}

\maketitle

\emph{Dedicated to Stanis{\l}aw Krajewski.}

\section{Krajewski's  Question}
\noindent At the Workshop on Formal Truth Theories in Warswaw, September 28--30, 2017, Stanis{\l}aw Krajewski asked
the following question.
\begin{quote}
Consider any theory $U$ of finite signature and suppose that
$U$ is not finitely axiomatizable. We  expand the language
of $U$ by finitely many extra predicate symbols. 
Can there be a finitely axiomatized $\alpha$ in the expanded language 
that conservatively extends
$U$ such that there is no finitely axiomatized $\beta$ that conservatively extends $U$
strictly $\dashv$-below $\alpha$, i.e.,
such that $ \alpha \vdash \beta$ and $\beta  \nvdash \alpha$? 
\end{quote}

\noindent
In this note, we prove that the answer is \emph{no} for all $U$ in case the expansion contains at least one symbol of arity 2 or larger. 
In fact, we will prove the desired result as one of several similar results, where, in stead of the relation of \emph{conservative extension}, we can also read
either \emph{extension} or \emph{model-conservative extension} and where $\dashv$ as ordering on finitely axiomatized extensions can be replaced by
some special kinds of interpretability, to wit: either
\emph{parameter-free interpretability that identically preserves the symbols of $U$} or \emph{interpretability with parameters that
identically preserves the symbols of $U$}.

For the case that we expand with only unary symbols, we provide a class of examples that illustrate that the answer may be \emph{yes} 
for the interpretability orderings. See Section~\ref{unapred}. We show, in a sense, that there is just one finite way to say that an ordering
is infinite in the language of linear orderings plus finitely many extra unary predicates. A consequence of our results is that
the finite model property for finitely axiomatized theories of linear order is decidable.

Kleene \cite{klee:fini52} and Craig {\&} Vaught \cite{crai:fini58} show that, in  case $U$ is recursively enumerable and has no finite models,
there is a finitely axiomatized $\alpha$ that conservatively extends
$U$ in the language expanded with at least one relation symbol of arity $\geq 2$. 
Craig {\&} Vaught prove an even stronger result where the finite theory $\alpha$ that extends $U$  is \emph{model-conservative} over $U$.
For completeness, we reprove the result by Craig {\&} Vaught below. See Section~\ref{existence}. 

Craig and Vaught provide an example that shows that, in the case of only unary expansion, there need not be
a finitely axiomatized conservative extension. This also follows by a result of Skolem that implies that
there is no formula in the language with identity and only unary predicate symbols that has
only infinite models ---if it has any models at all.

In Appendix~\ref{environment}, we provide some results in the environment of our problem.

\section{Preliminaries}
We work in languages with only relational symbols. This restriction is not really a limitation since we 
can simulate the presence of terms using the well-known term-unwinding algorithm. We
work in languages with identity as logical symbol.

Suppose $U$ is a theory of finite signature $\Sigma_U$ and let $\Theta$ be a finite signature disjoint from $\Sigma_U$.
We use $A,B, \ldots$ for sentences of signature $\Sigma_U$ and $\alpha,\beta, \ldots$ for sentences of signature $\Sigma_U+\Theta$. 
We will confuse sentences with finitely axiomatized
theories. Thus, we will write, e.g., both $\alpha +A$ and $(\alpha \wedge A)$.

We write ${\sf M}_\Theta$ for the maximal arity of a symbol in $\Theta$ and $|\Theta|$ for the cardinality of $\Theta$.

\subsection{Theories}
We will employ a number of specific theories in our paper.

\medent
The theory {\sf INF} is the theory in the language of identity that has, for every $n$, an axiom that says `there are at least $n$ elements', or,
\[\exists x_0\ldots \exists x_{n-1} \bigwedge_{i<j<n} x_i \neq x_j.\] 

\medent
The theory {\sf LIN} is the theory of linear order.

\medent
The theory ${\sf S}^1_2$ is the weak arithmetic of p-time computability. See, e.g.,
\cite{buss:boun86} or \cite{haje:meta91} for a description.

\medent
The theory {\sf AS} or Adjunctive Set Theory has the following axioms.
\begin{enumerate}[{\sf AS}1.]
\item
$\exists x\, \forall y\, y \not\in x$
\item
$\forall x\, \forall y\, \exists u\, \forall v\, (v\in u \iff (v\in x\vee v=y))$.
\end{enumerate}
We refer the reader to \cite{viss:what13} for further information about {\sf AS}.

We can interpret ${\sf S}^1_2$ in {\sf AS}. We fix one such interpretation $N$.
We can arrange it so that assignments for formulas coded in $N$ have desirable 
properties that make a definition of satisfaction meaningful. We can also arrange that, in the
obvious interpretation of {\sf AS} in the hereditarily finite sets, $\mathbb{HF}$, the interpretation of $N$ is the standard numbers.

\medent
Let a signature $\Sigma$ be given. We define the finitely axiomatized theory ${\sf C}_\Sigma$ in the $\Sigma$-language expanded with one
fresh binary relation symbol $R$ as follows.
Let $x \in y :\iff xR y \wedge \neg\, yRy$ and $\sat(x,y) := \exists z\, (zRz \wedge  \tupel{x,y} R z)$. Here
$\tupel{\cdot,\cdot}$ is the Kuratowski  $\in$-pairing.
We take the following axioms for ${\sf C}_\Sigma$.
\begin{itemize}
\item
{\sf AS} plus extensionality for $\in$.
\item
$\sat$ is a satisfaction predicate for the $\Sigma$-language
 with commutation conditions for $\Sigma$-formulas coded in $N$. 
\end{itemize}

\noindent 
It is easy to see that every countably infinite model of signature $\Sigma$ can be expanded
to a model of ${\sf C}_\Sigma$. We will write ${\sf C}_U$ for ${\sf C}_{\Sigma_U}$.

\begin{remark}
An alternative way to construct a functional equivalent of ${\sf C}_\Sigma$ is to expand $\Sigma$ with two primitive
predicates $\in$ and {\sat} and 
employ a theorem of Tarski to reduce the two predicates to one in a definitionally equivalent way.
See \cite{tars:gene54}. This strategy is employed by Craig {\&} Vaught in \cite{crai:fini58}. 
\end{remark}

\noindent
The theory {\tinyset}, or, \emph{tiny set theory},  is defined as follows.
\begin{enumerate}[{\tinyset}1.]
\item Extensionality:\\
$\forall x\,\forall y \, (x=y \iff \forall z\, (z\in x \iff z\in y))$. 
\item
Restricted Adjunction of Elements:\\
$\forall x\, \forall y\,\forall z\, (y\in z \to \exists u\, \forall v\, (v \in u \iff (v\in x \vee v=y)))$.
\item
Foundation for Sets:\\
$\forall x\, \forall y\, (y\in x \to \exists z\, (z \in x \wedge \forall v\, (v\in z \to v\not\in x)))$.
\end{enumerate}

\noindent
We have that following insights.

\begin{theorem}
Let $\mathcal M$ be a finite model of {\tinyset}. Then, $\mathcal M$ is
isomorphic with $\tupel{\wp X, \in}$, where $X$ is a pure finite transitive
 set. 
\end{theorem}

\begin{proof}
Consider a finite model $\mathcal M$ of 
{\tinyset}.  Let $m$ be a set in $\mathcal M$ with the
maximal number of elements. By {\tinyset}2, it follows that
every element of some other set is in $m$. Thus, either $m$ is empty or its
$\in$-minimal element, guaranteed by {\tinyset}3, is empty.
Given that we have the empty set, we can use {\tinyset 2} to build any
subset of $m$. Clearly, from the external point of view, the relation $\in^{\mathcal M}$ restricted to the set of all $k$ such
that $k \in^{\mathcal M}m$ is a transitive, well-founded and extensional binary relation.

Thus, the set of $k \in^{\mathcal M}m$ equipped with $\in^{\mathcal M}$ is isomorphic to
$X$ equipped with $\in$, for some 
finite transitive set $X$. Hence, the model $\mathcal M$ is isomorphic to the powerset of $X$ equipped with $\in$.
\end{proof}

\begin{theorem}
Suppose $S$ is a $\Sigma^0_1$-sentence. There is a translation $S \mapsto \widetilde S$ of the arithmetical to the set-theoretical
 set-theoretical $\Sigma_1$-formulas, such that:
 $S$ is true iff, there is a pure transitive, hereditarily finite set $X$, such that $\tupel{\wp X, \in} \models \widetilde S$.  
\end{theorem}

\begin{proof}
The proof is basically a careful translation of $S$ to the set theoretical language. One way of doing that is by using the realization in
Barwise's book \cite{barw:admi17}.  Barwise develops the theory of $\Sigma$-definability in {\sf KPU}.
We note that  $\mathbb{HF}$ is a model of {\sf KPU}. First Barwise proves that any $\Sigma$-formula is equivalent to a $\Sigma_1$-formula (Theorem I.4.3). 
He gives $\Delta$-definitions for predicates ${\sf Nat}(x)$ and $ \leq $ (Section I.5). Also he shows that the functions 
{\sf S} (Section I.5), $+$ and $\times$ are $\Sigma$-functions (Section I.6). 
Finally, he proves that, if we extend the language of {\sf KPU} by $\Delta$-predicates (Lemma I.5.2) and $\Sigma$-functions (Lemma I.5.4),
 then the extension will be conservative and each $\Sigma$-formula of the extended language will be equivalent to a $\Sigma$-formula of the original. 

The translation that we need is constructed from Barwise's results as follows. We first translate $S$ to a pure $\Sigma^0_1$-formula $S'$ where all 
bounding terms are variables. See, e.g., \cite[Section 3]{viss:onq17} for a sketch of how to do that. Then, we translate $S'$
 to a $\Sigma$-formula $S''$ in {\sf KPU} conservatively extended by {\sf Nat}, $\leq$, {\sf S}, $+$, and $\times$. 
  It is obvious that this translated formula is true in $\mathbb{HF}$ (extended by definitions) iff the original formula were true on standard natural numbers. 
  Then, we obtain an equivalent $\Sigma$-formula $S'''$ of pure {\sf KPU} and, finally, we get an equivalent $\Sigma_1$-formula $\widetilde S$.

\medent
Alternatively, we could extract a mapping from the proof of  \cite[Theorem I.2.3]{barw:admi17}.  
\end{proof}

\noindent
The finite models of {\tinyset} can only have cardinalities that are powers of 2. This is too restrictive for our purposes.
We will need that, if our small set theory plus $\widetilde S$ has a finite model, then it has a model of any
greater cardinality and that the construction of this larger model  is sufficiently uniform. There are many ways of achieving this. We sketch two of them. The first is
to allow elements with loops in the domain. We replace {\tinyset} by its relativization to the
elements $a$ such that $a \not\in a$. Alternatively, we replace identity by extensional equivalence in the axioms.
We adopt this second strategy and describe it in a bit more detail.

Let us write $a \approx b$ for $\forall x\, (x\in a \iff x \in b)$. We now define the theory:
\begin{enumerate}[$\tinyset^\star$1.]
\item
Congruence:\\
$\forall x\,\forall x'\, \forall y \, \forall y'\, ((x \approx x' \wedge y \approx y') \to (x\in y \iff x'\in y')) $.
\item
Restricted Adjunction of Elements:\\
$\forall x\, \forall y\,\forall z\, (y\in z \to \exists u\, \forall v\, (v \in u \iff (v\in x \vee v \approx y)))$.
\item
Foundation for Sets:\\
$\forall x\, \forall y\, (y\in x \to \exists z\, (z \in x \wedge \forall v\, (v\in z \to v\not\in x)))$.
\end{enumerate}

\noindent
We modify our mapping $S\mapsto \widetilde S$ 
by replacing $=$ by $\approx$. Say the resulting formula is $\breve S$.  We write $[S]$ for $\tinyset^\star+ \breve S$. 
It is clear that any model of $[S]$ can be modified into an $[S]$-model of greater cardinality simply by adding extra elements to
a $\approx$-equivalence class. We will return to this idea in the proof of Theorem~\ref{grappasmurf}.

\begin{remark}
Our use of $\tinyset$ and $\tinyset^\star$ was inspired by Harvey Friedman's use of theories of a number in \cite{frie:inte07}.
\end{remark}

\subsection{Relations between Theories}
Suppose $\Sigma$ and $\Theta$ are two disjoint finite signatures. A $\Sigma,\Theta$-translation
$\tau$ is given by a number $p$ and a mapping that sends the predicate symbols $P$ of $\Theta$ of arity $n$ to
a formula $\pi(\vec v, \vec w)$ where the $\vec v,\vec w$ are pairwise-disjoint designated variables and $\vec v$ has 
length $n$ and $\vec w$ has length $p$. The translation $\tau$ can be lifted to all  $\Sigma,\Theta$-formulas in which the $\vec w$ do not occur as
free variables as follows:
\begin{itemize}
\item
If $P$ is a symbol from $\Sigma$:
$P^\tau(\vec x) : \iff \tau(P)(\vec v,\vec w)[\vec v := \vec x]$, where we rename bound variables in $\tau(P)$ if
in case they obstruct substitutability of the $\vec x$.
\item
If $Q$ is a symbol from $\Theta$:
$Q^\tau(\vec x) : \iff Q(\vec x)$,
\item
$(\cdot)^\tau$ commutes with the logical connectives where we rename variables of quantifiers if
they would bind any of the $\vec w$.  
\end{itemize} 
Let $V$ and $W$ be $\Sigma,\Theta$-theories.
An $\Sigma,\Theta$-interpretation $\mathfrak t:V\to W$, is given by
a $\Sigma,\Theta$-translation $\tau$ and a parameter domain ${\sf par}_{\mathfrak t}(\vec w)$ in the $\Sigma,\Theta$-language.
We demand that $W \vdash \exists \vec w\, {\sf par}_{\mathfrak t}(\vec w)$ and that, for all $\Sigma,\Theta$-sentences $\alpha$, if $V\vdash \alpha$,
then  $W \vdash \forall \vec w\, ({\sf par}_{\mathfrak t}(\vec w) \to \alpha^{\tau})$. An interpretation is \emph{parameter-free} if the
dimension $p$ of the parameter domain associated to its translation $\tau$ is zero. In case of parameter-free interpretations, we will
simply omit the parameter domain ---since it is $\top$ modulo $W$-provable equivalence.

We write $W \jump V$ (or, more officially, $W \jump_{\Sigma,\Theta} V$) if there is an $\Sigma,\Theta$-interpretation $\mathfrak t: V \to W$. 
We will always suppress the $\Sigma,\Theta$-subscript since the relevant pair of signatures is, in all cases, contextually given.

We write  $W \jump_{\sf pf} V$ if there
is a parameter-free interpretation $\mathfrak t: V \to W$. 

We write $\alpha^{\mathfrak t}$ for $\alpha^{\tau_{\mathfrak t}}$, etcetera.

\begin{remark}
Kentaro Fujimoto  in his paper \cite{fuji:rela10} introduced the notion of \emph{relative truth definability}. This notion is our notion $\jump_{\sf pf}$
restricted to expansions with a truth predicate.
\end{remark}

\noindent
We have the following small insight.

\begin{theorem}\label{grappasmurf}
Consider any finitely axiomatized theory $\beta$.
Then, $W \jump \beta$ iff, for some translation $\tau$, we have $W \vdash \exists \vec w\, \beta^{\tau}$.
\end{theorem}

\begin{proof}
From left to right is immediate. From right to left, we take as parameter-domain  $\beta^\tau$.
\end{proof}

\noindent
As a consequence it suffices, in case the target theory of an interpretability claim is finitely axiomatized,
to just specify the translation.

\begin{remark}
Mycielski, Pudl\'ak and Stern, in their fundamental  paper \cite{myci:latt90}, \emph{define} interpretability with parameters in the
style of Theorem~\ref{grappasmurf} without a parameter-domain. 
\end{remark}

\medent
We will need the following insight.

\begin{theorem}\label{hulpsmurf}
Consider any true $\Sigma^0_1$-sentence $S$.
Suppose $\mathcal M$ is a model of $\tinyset + \widetilde S^N$ of cardinality $n$.
Then, $\exists x_0\, \ldots, x_{n-1}\,\bigwedge_{i<j<n} x_i\neq x_j \jump [S]$.
\end{theorem} 

\begin{proof}
Suppose the domain of $\mathcal M$ is $m_0,\ldots,m_{n-1}$.
We take as parameter-domain $\bigwedge_{i<j<n} w_i \neq w_j$.
We take as translation:
\begin{itemize}
\item
$x \sim y : \iff  \bigwedge_{1< i < n}(x \neq w_i \wedge y \neq w_i) \vee \bigvee_{1<i  < n} (x=w_i \wedge y = w_i)$.
\item 
$v_0\mathrel{R_{\vec w}} v_1 := \bigvee \verz{(v_0 \sim w_i \wedge v_1\sim w_j) \mid i,j<n\text{ and } R^{\mathcal M}(d_i,d_j)} $
\end{itemize}
It is easy to see that the resulting parameter-domain plus translation witness  $\exists x_0\, \ldots, x_{n-1}\,\bigwedge_{i<j<n} x_i\neq x_j \jump [S]$.
\end{proof}

\noindent
Here are some further notions.
\begin{itemize}
\item
$U \dashv_{\sf c} \alpha$ iff $\alpha \vdash U$ and, for all $A$, if $\alpha \vdash A$, then $U \vdash A$. In case $U \dashv_{\sf c} \alpha $, we say that
$\alpha$   is \emph{a  conservative extension of $U$}.
\item
$U \dashv_{\sf mc} \alpha$ iff $\alpha \vdash U$ and every model of $U$ can be expanded to a model of $\alpha$. 
In case $\alpha \dashv_{\sf mc} U$, we say that
$\alpha$   is \emph{a  model-conservative extension of $U$}.
\item
$\alpha \vdashneq \beta$ iff $\alpha \vdash \beta$ and $\beta \nvdash \alpha$.
\item
$\alpha \jumpneq_{\sf pf} \beta$ iff $\alpha \jump_{\sf pf} \beta$ and $\beta \not \jump_{\sf pf} \alpha$.
\item
$\alpha \jumpneq \beta$ iff $\alpha \jump \beta$ and $\beta \not \jump \alpha$.
\end{itemize}

\noindent 
We note the following useful fact.

\begin{fact}
Let $\sqsubseteq$ be one of $\dashv$, $\dashv_{\sf c}$ or $\dashv_{\sf mc}$.
Let $\preceq$ be one of $\dashv$, $\jumpb_{\sf pf}$ or $\jumpb$.
Suppose $U \sqsubseteq \alpha$ and $U \dashv \beta \preceq \alpha$.
Then, $U \sqsubseteq \beta$.
\end{fact}

\section{Existence}\label{existence}
In this section, we provide a proof of the Kleene-Craig-Vaught result. We also provide some insights in its
direct neighborhood.

\begin{theorem}[Kleene-Craig-Vaught]\label{kcv}
Suppose $U$ is a recursively enumerable theory without finite models. 
We expand the signature of $U$ with $\Theta$ with ${\sf M}_\Theta \geq 2$.
Then, there is an $\alpha$ such that $U \dashv_{\sf c} \alpha$. 
\end{theorem}

\noindent
We note that if $U$ has a finite conservative extension $\alpha$, then $U$ must be recursively enumerable. So, the existence theorem is best possible as far as
the complexity of $U$ is concerned.

\begin{proof}
We first treat the case where we expand $U$ with a binary relation symbol $R$.
Let $S(x)$ be a $\Sigma^0_1$-formula that represents the theorems of $U$. Let {\sf true} be the truth-predicate that is based
on the satisfaction predicate {\sat} of the theory ${\sf C}_U$.
Let $\alpha := ({\sf C}_U \wedge \forall A\in {\sf sent}_U^N\, (S^N(A) \to {\sf true}(A)))$.
Clearly $\alpha \vdash U$. 

Consider any model $\mathcal M$ of $U$. Since the domain of $\mathcal M$
 is infinite, we can clearly expand this to a model that is $\mathbb{HF}$ as interpretation of $\in$ and that yields modulo arithmetization
 the standard satisfaction predicate as interpretation of \sat.
 
 \medent
 We extend the result to any $\Theta$ with ${\sf M}_\Theta\geq 2$, by using any predicate $P(\vec x\,)$ of $\Theta$ of arity $\geq 2$ as 
 replacement of $R$, using the first two argument places to mimic the argument places of $R$ and treating the remaining ones as dummies.
 We also treat the remaining predicate symbols of $\Theta$ as don't care.
\end{proof}

\noindent
We show how to extend Theorem~\ref{kcv} to the case where $U$ has only finitely many finite models modulo isomorphism.

\begin{lemma}\label{knutselsmurf}
Let $U$ be a recursively enumerable theory. 
We expand the signature of $U$ with $\Theta$ with ${\sf M}_\Theta \geq 2$.
Suppose that, for some $A$, we have  $U \dashv A$ and $U+ \neg \, A$ has no finite models. 
Then, $U$ has a finitely axiomatized
model-conservative extension $\alpha$ in the expanded language.
\end{lemma}

\begin{proof}
We apply Theorem~\ref{kcv} to $U+\neg\, A$. This gives us a $\beta$ with $U +\neg \, A \dashv_{\sf mc} \beta$.

We show that $U \dashv_{\sf mc} (A \vee \beta) =: \alpha$.
Clearly, $A \vee \beta \vdash U$.
Let $\mathcal M$ be a model of $U$. In case $\mathcal M \models A$, we are done. In case $\mathcal M \not\models A$,
we have $\mathcal M \models U+\neg\, A$. 
So, we can expand $\mathcal M$ to a model of $\beta$, and we are, again, done.
\end{proof}

\begin{theorem}\label{existencetwee}
Suppose $U$ is a recursively enumerable theory and suppose $U$ has only finitely many finite models modulo isomorphism. 
We expand the signature of $U$ with $\Theta$ with ${\sf M}_\Theta \geq 2$.
Then, for some $\alpha$, we have   $U\dashv_{\sf mc} \alpha$. 
\end{theorem}

\begin{proof}
We take $A$ the disjunction of the model descriptions of the finite models of $U$. It is easily seen that
$A$ satisfies the conditions of Lemma~\ref{knutselsmurf}.
\end{proof}

\noindent
We note that the set of finite models of a recursively enumerable theory might be complete $\Pi^0_1$.
On the other hand, whenever $U\vdash_{\sf c} \alpha$, the set of finite models for $U$ is NP, by Fagin's theorem.
See e.g. \cite{imme:desc12}. So, not all recursively enumerable $U$ can have a finite extension $\alpha$ in an
expanded language such that $U\dashv_{\sf c} \alpha$. 

\begin{question}\label{Q1}
It there a recursively enumerable theory $U$ with an NP set of finite models (modulo isomorphism)
such that there is no $\alpha$ with $U \dashv_{\sf c} \alpha$?
\end{question}

\begin{question}\label{Q2}
Can we find a recursively enumerable $U$ and an $\alpha$ in an expanded language, such that $U \dashv_{\sf c} \alpha$,
where there is no $\beta$ such that $U \dashv_{\sf mc} \beta$?
\end{question}

\noindent
In case we work with languages without identity symbol, the situation changes, since we could have
$U \dashv_{\sf c} \alpha$, where $U$ has finite models and $\alpha$ has not. In fact, by slightly modifying
the proof of Theorem~\ref{kcv}, we can, in the identity-free case, find an $\alpha$ such that $U \dashv_{\sf c} \alpha$,
for any recursively enumerable $U$.
We note that this last observation does not hold for the case of $\dashv_{\sf mc}$. In this case the application of
Fagin's Theorem still obtains.

\medent
In case we expand the language of $U$ only with unary predicates,
we need not be able to find an $\alpha$ such that $U \dashv_{\sf c} \alpha$, as shown in \cite{crai:fini58}.

\medent
Finally, we consider what happens when we consider the relation $ \dashv $ in stead of $\dashv_{\sf c}$ and
$\dashv_{\sf mc}$. It is clear that we can always find an $\alpha$ such that $U \dashv \alpha$, to wit $\alpha := \bot$.
As a consolation, for those who find this example too trifling, we have the following result.

\begin{theorem}\label{tuinsmurf}
Suppose $U$ is a consistent recursively enumerable theory. 
We expand the signature of $U$ with a binary relation symbol $R$.
Then, there is a consistent $\alpha$ in the expanded language such that $U \dashv \alpha$. 
\end{theorem}

\begin{proof}
Our theorem is a direct consequence of Theorem~\ref{grotesmurf}
applied to the relations $\dashv$ of $\sqsubseteq$ and $\dashv$ for $\preceq$ and $\bot$ as
initial example of $U \dashv \bot$. 

However, we can also reason as follows.
Suppose $U+ {\sf INF}$ is consistent. In that case we can apply Theorem~\ref{kcv} to $U+{\sf INT}$ to obtain the desired consistent
$\alpha$. If  $U+ {\sf INF}$ is inconsistent, $U$ clearly only has models of size $\leq n$, for some $n$. In this case,
$U$ can be axiomatized by the disjunction of the model descriptions of its finite models. We can now take this disjunction as our
$\alpha$.
\end{proof}

\section{The Main Theorem}
We formulate our main theorem.

\begin{theorem}\label{heiligesmurf}
Consider any recursively enumerable theory $U$ that is not finitely axiomatizable and any finite expansion $\Theta$ of the signature of $U$ with ${\sf M}_\Theta \geq 2$. 
Then, for all $U \dashv \alpha$, there
is a $\beta$ with $U \dashv \beta \jumpbneq \alpha$.
\end{theorem}

\noindent
We note that Theorem~\ref{heiligesmurf} is not a direct answer to Krajewski's question. However, we will show, in Corollary~\ref{goedesmurf},
that the negative answer to Krajewski's question is an immediate consequence of Theorem~\ref{heiligesmurf}.

\begin{proof}[Proof of Theorem~\ref{heiligesmurf}]
We first treat the case where $\Theta$ consists of a single binary relation symbol $R$. At the end of the proof we will describe how to adapt the argument
to the more general case.

\medent
Suppose $U$ is recursively enumerable and not finitely axiomatizable.
We split the proof in two cases:
\begin{enumerate}[A.]
\item
There is no finitely axiomatizable sub-theory $U_0$ of $U$ such that $U_0+{\sf INF} \vdash U$.
\item
There is a finitely axiomatizable sub-theory $U_0$ of $U$ such that $U_0+{\sf INF} \vdash U$.
\end{enumerate}

\medent
\emph{We treat case \textup(A\textup).} 
By Craig's trick, we can find a $\Delta_0(\omega_1)$ axiomatization ${\sf ax}_U$ of $U$. Let {\sf true} be the truth predicate derived from {\sat} of ${\sf C}_U$.

Suppose $U \dashv \alpha$.
We find a fixed point $K$ such that, in the standard model, we have $K$ iff $\gamma \jump \alpha$, where  $K = \exists x\, K_0(x)$ for $K_0 \in \Sigma_1^{\sf b}$ and
\[ \gamma :\iff {\sf C}_U \wedge \exists x\in N \, (K^N_0(x) \wedge \forall y <^N x\, ({\sf ax}^N_U(y) \to {\sf true}(x))).\]

 \noindent
Our first order of business is to prove that $K$ is false. 
 
 Suppose $K$ were true. Let it be witnessed by $k$. Then,
 we have ${\sf C}_U \vdash (K_0( \underline k))^N$. Using the commutation conditions it follows
 that: \[U+{\sf C}_U \vdash \forall y <^N \underline k\, ({\sf ax}^N_U(y) \to {\sf true}(x)).\] Thus,
 $U+{\sf C}_U \vdash \gamma$. By compactness, for some finite sub-theory $U_0$ of $U$, we have
 $U_0+{\sf C}_U \vdash \gamma$.
 
 Since $K$ is true we have $\gamma \jump \alpha$. Suppose we have $U \vdash B$, for any $U$-sentence $B$. Then, $\alpha \vdash B$, and thus,
 $\gamma \jump B$. Since $B$ is in the $U$-language, we find $\gamma \vdash B$. 
Since $B$ was an arbitrary consequence of $U$, we find $\gamma \vdash U$ and, hence, $U_0+{\sf C}_U \vdash U$.
 Since, every model  of $U_0+{\sf INF}$ can be expanded to a model
 of ${\sf C}_U$, we have $U_0+{\sf INF} \vdash U$, \emph{quod non}, by assumption (A).
 
 \medent
 We have shown that $\gamma \not \jump \alpha$. The falsity of $K$ also implies that $\gamma \vdash U$, since $\gamma$ knows of every standard
 number that it is not a witness of $K$.
 It follows that $U \dashv (\alpha \vee \gamma) \jumpb \alpha$ and $(\alpha \vee \gamma) \not \jump \alpha$.
 So,  $U \dashv (\alpha \vee \gamma) \jumpbneq \alpha$.
 Thus, we can take $\beta := (\alpha \vee \gamma)$.
 
 \medent
\emph{We treat case \textup(B\textup).}
We suppose, in order to obtain a contradiction, that $\alpha^\star$ is $\jumpb$-minimal such that $U \dashv \alpha^\star$.

Let $A$ be a sentence axiomatizing a finitely axiomatizable sub-theory $U_0$ of $U$ such that $U_0 +{\sf INF} \vdash U$.
 We find, using the G\"odel Fixed Point Lemma, a $\Sigma^0_1$-sentence $K$ such that, in the standard model,
$K$ iff $(A \wedge [K]) \jump \alpha^\star$. Here we use $R$ in the role of $\in$.

\medent
We claim that $K$ is false. Suppose $K$ is true. Then, we have a finite model of $\tinyset + \widetilde S^N$ of size, say, $n$.
 By Theorem~\ref{hulpsmurf}, we have:
\[ (A + \exists x_0\cdots \exists x_{n-1}\, \bigwedge_{i<j<n} x_i\neq x_j) \jump (A+[K]).\]
Since $(A \wedge [K]) \jump \alpha^\star$, we find:
\[(\dag)\;\;\;\; (A + \exists x_0\cdots \exists x_{n-1}\, \bigwedge_{i<j<n} x_i\neq x_j) \jump \alpha^\star.\] 
Let us write ${\sf D}_{\mathcal M}$ for the $U$-sentence describing the existence of a finite $\Sigma_U$-model $\mathcal M$.
We have: 
\[U+  \forall x_0\cdots\forall x_{n-1}\bigvee_{i<j<n} x_i = x_j \vdash  \bigvee\verz{{\sf D}_{\mathcal M} \mid \text{$\mathcal M$
 is a $U$-model of cardinality $<n$}} \] 
Since ${\sf D}_\mathcal M \vdash U$, it 
follows, by the $\jumpb$-minimality of $\alpha^\star$, that ${\sf D}_{\mathcal M} \jump \alpha^\star$. Ergo,
\[(\ddag)\;\;\;\; (U+  \forall x_0\cdots\forall x_{n-1}\bigvee_{i<j<n} x_i = x_j) \jump \alpha^\star.\]
We may conclude from (\dag) and (\ddag) that $U \jump \alpha^\star$.

We find that, for some finitely axiomatized $U_1\subseteq U$, we have $U_1 \jump \alpha^\star\vdash U$.
It follows that
$U_1 \vdash U$, in contradiction to the fact that $U$ is not finitely axiomatizable. 

Thus, $K$ is false.

\medent
Since $K$ is false, we have $(A \wedge [K]) \not \jump \alpha^\star$. 
Moreover, $A+[K] \vdash A + {\sf INF} \vdash U$.
It follows that $\beta := (\alpha^\star \vee (A \wedge [K]))$  is
$\jumpbneq$-below $\alpha^\star$ and $\dashv$-above $U$. A contradiction.
 
 \medent
 We can easily extend our proof to the general case where $\Theta$ is finite and ${\sf M}_\Theta \geq 2$.
 We simply choose one predicate symbol $P$ with arity $\geq 2$. We use the first two argument places of $P$ to simulate the
 argument places of $R$. The remaining argument places are treated as dummy variables. All other predicate symbols in $\Theta$ are don't care.
\end{proof}

\noindent
We can now answer Krajewski's question.

\begin{corollary}\label{goedesmurf}
Consider any theory $U$ that is not finitely axiomatizable and any finite expansion $\Theta$ of the signature of $U$ with ${\sf M}_\Theta \geq 2$. 
Then, for all $\alpha$ with $U \dashv_{\sf c} \alpha$, there
is a $\beta$ with $U \dashv_{\sf c} \beta \dashvneq \alpha$.
\end{corollary}

\noindent
We note that, unlike in the case of Theorem~\ref{heiligesmurf}, we do not need to demand that $U$ is recursively enumerable. 
If the theory has a finite conservative extension at all, then the theory is
automatically recursively enumerable.

\begin{proof}
Suppose $U$ that is not finitely axiomatizable. Let $\Theta$ be an expansion of the signature of $U$ with ${\sf M}_\Theta \geq 2$. 
Suppose $U \dashv_{\sf c} \alpha$. It follows that $U$ is recursively enumerable and $U \dashv\alpha$, and, hence, by Theorem~\ref{heiligesmurf}, there
is a $\gamma$ with $U \dashv \gamma \jumpbneq \alpha$. Let $\beta := (\alpha \vee \gamma)$. Clearly, $U \dashv \beta \dashv \alpha$.
Suppose that we would have $\beta \vdash \alpha$. Then, it would follow that $\gamma \vdash \alpha$. From this, we get $\gamma \jump \alpha$.
Quod non. Hence $\beta \dashvneq \alpha$.   Finally, since a sub-theory of a conservative theory is conservative, we have $U \dashv_{\sf c} \beta$.
\end{proof}

\noindent
In fact we can do more than Corollary~\ref{goedesmurf}. We can replace $\dashv$ in Theorem~\ref{heiligesmurf} by any of $\dashv_{\sf c}$ or $\dashv_{\sf mc}$ and
we can replace $\jumpb$ by any of $\dashv$ or $\jumpb_{\sf pf}$. In all these cases, we obtain a valid theorem. 
The reasoning for the seven further cases is fully analogous to the reasoning for Theorem~\ref{goedesmurf}.
We spell this out in Appendix~\ref{ninecases}. 

\medent
We end this Section with some questions.

\begin{question}\label{Q3}
Consider any theory $U$ that is not finitely axiomatizable and any finite expansion $\Theta$ of the signature of $U$ with ${\sf M}_\Theta \geq 2$. 

Suppose $\alpha$ is not interpretable (in the full sense of interpretability) in $U$. Is there an extension $\beta$ of $U$, such that $\alpha$ is not interpretable in $\beta$?
\end{question}

\begin{question}\label{Q4}
Consider any theory $U$ that is not finitely axiomatizable and any finite expansion $\Theta$ of the signature of $U$ with ${\sf M}_\Theta \geq 2$. 

We take as the length of a proof the number of symbols in the proof written in a fixed finite alphabet. 
We define $\beta \preceq_{\sf sp} \alpha$ iff there is a polynomial $P(x)$ such that, for every $A$ (of the language of $U$),
 if $A$ is provable from $\beta$ by a proof of the length $n$, then $A$ is provable from $\alpha$ by a proof of  length $\leq P(n)$.
 
Can there be a $ \preceq_{\sf sp}$-minimal conservative extension $\alpha^\star$ of $U$? 
\end{question}

\section{The Case of Unary Predicates}\label{unapred}
In this section, we provide an example of a class of theories $U$ such that, in the language
expanded with a non-empty finite signature of unary predicate symbols, we have a $\jumpb_{\sf pf}$-minimal $\alpha^\star$
such that $U \dashv_{\sf mc} \alpha^\star$.

\medent
Let {\sf LIN} be the theory of linear order. 
Suppose $\Theta$ is a finite signature consisting of unary predicate symbols
and let $P$ be a designated symbol in $\Theta$. We take
$\Lambda$ to be the signature of the theory of order extended with $\Theta$. 
Let ${\sf UB}(P)$  be the property of $P$ that $P$ defines a non-empty set that
either does not have a minimal element or does not have a maximal element. More formally:
\begin{itemize}
\item
 ${\sf UB}(P) :\iff \exists x\, P(x) \; \wedge \\
 \hspace*{1cm} (\forall y\, (P(y) \to \exists z\, (P(z) \wedge y <z)) \vee 
\forall y\, (P(y) \to \exists z\, (P(z) \wedge z < y)))$.
 \end{itemize}
 
 \medent
The main theorem of this section is as follows.
\begin{theorem}\label{nerdsmurf}
Suppose $\Theta$ is a non-empty finite signature consisting of unary predicate symbols
and let $P$ be in $\Theta$. 
Suppose $A$ is a finite extension of {\sf LIN} and let 
$U:=  A + {\sf INF}$. We take $\alpha^\star := (A\wedge {\sf UB}(P))$.
We have:
\begin{enumerate}[a.]
\item
 $U\dashv_{\sf mc} \alpha^\star$.
 \item
For all $\alpha$ such that $U \dashv \alpha$, we have $\alpha^\star \jumpb_{\sf pf} \alpha$.
\end{enumerate}
\end{theorem}

\noindent
So, in a sense, the theorem tells us that, in the extended language, there is only one finite way to say that 
we exclude finite models of $A$. We note that Theorem~\ref{nerdsmurf} with $\jumpb_{\sf pf}$ replaced by
$\jumpb$ follows from  Theorem~\ref{nerdsmurf}. However the following is open.

\begin{question}\label{Q5}
Is there an example of a theory $U$ and an $\alpha$ in the language of $U$ extended with a non-empty finite signature of unary predicate
symbols, such that $\alpha$ is a $\dashv$-minimal conservative extension of $U$?
\end{question}

\noindent
We prove (a) of Theorem~\ref{nerdsmurf} now and postpone the proof of (b) until we have done
some preparatory work.

\begin{proof}[Proof of Theorem~\ref{nerdsmurf}\textup(a\textup)]
Consider any infinite  model $\mathcal N$ of $A$. 
Since $\mathcal N$ cannot be both well-founded and converse well-founded,
we can find the desired interpretation of $P$.
\end{proof}

\noindent
We start our prepratory work with a theorem on linear orderings. Consider the theory {\sf LIN} of linear order. We extend the signature of {\sf LIN}. 
 with a  signature $\Theta$ consisting of finitely many unary predicate symbols. Say the resulting signature is $\Lambda$.
 As usual, we let $\alpha,\beta,\ldots$ range over $\Lambda$-language
We add two two unary predicates $\apr_0$ and $\apr_1$ to the $\Lambda$-language.
We write $\alpha^{\apr_i}$ for the result of relativizing all quantifiers in $\alpha$ to 
$\apr_i$. 
We add the following axioms to {\sf LIN}:
\begin{itemize}
\item
$\forall x\, (\apr_0(x) \iff \neg\, \apr_1(x))$.
\item
$\forall x\,\forall y\, ((\apr_0(x) \wedge y<x) \to \apr_0(y))$.
\end{itemize}
Note that we allow the $\apr_i$ to be empty. Relativization to the empty domain is as expected: an existential sentence relativized to the empty domain is false and
a universal one is true.

Say the resulting theory (in the language of signature $\Lambda(\apr_0,\apr_1)$) is ${\sf LIN}^{\sf split}$. Consider a formula $\alpha$. Let $\vec x;\vec y$ be a partition in two
parts of a finite set of variables. Let the context ${\sf C}(\vec x;\vec y)$ be a conjunction of all formulas $\apr_0(x_i)$ and $\apr_1(y_j)$.  

\begin{theorem}\label{basicsmurf}
Consider a formula $\alpha(\vec x,\vec y)$ with all free variables shown. Then, over ${\sf LIN}^{\sf split}+{\sf C}(\vec x;\vec y)$, the formula
 $\alpha(\vec x,\vec y)$ is equivalent to a boolean combination of formulas of the form $\eta^{\apr_0}(\vec x)$ \textup(all free variables shown\textup) and
 $\theta^{\apr_1}(\vec y)$ \textup(all free variables shown\textup).
\end{theorem}

\begin{proof}
The proof is by induction on $\alpha$. For the atomic case, we have:
\begin{itemize}
\item
${\sf LIN}^{\sf split} + \apr_0(x)\wedge \apr_1(y) \vdash x< y \iff \top$.
\item 
${\sf LIN}^{\sf split} + \apr_0(x)\wedge \apr_1(y) \vdash y< x \iff \bot$.
\item
${\sf LIN}^{\sf split} + \apr_0(x_0)\wedge \apr_0(x_1) \vdash x_0< x_1 \iff x_0 < x_1$.
\item
${\sf LIN}^{\sf split} + \apr_1(y_0)\wedge \apr_1(y_1) \vdash y_0< y_1 \iff y_0 < y_1$.
\item
${\sf LIN}^{\sf split} + \apr_0(x) \vdash P(x) \iff P(x)$.
\item
${\sf LIN}^{\sf split} + \apr_1(y) \vdash P(y) \iff P(y)$.
\item
Similary, for further unary predicates.
\end{itemize}

\noindent
Preservation of the desired property under the propositional connectives is trivial. We treat the case of the existential quantifier.
Suppose $\alpha = \exists u\, \alpha_0(u,\vec x,\vec y)$. After some rewriting we have:
\begin{itemize}
\item
${\sf LIN}^{\sf split} + {\sf C}(u,\vec x;\vec y) \vdash \alpha_0(u,\vec x,\vec y) \iff \bigvee_{i<n}(\eta^{\apr_0}_{i}(u,\vec x) \wedge \theta_i^{\apr_1}(\vec y)) $.
\item
${\sf LIN}^{\sf split} + {\sf C}(\vec x;u,\vec y) \vdash \alpha_0(u,\vec x,\vec y) \iff \bigvee_{j<m}(\kappa^{\apr_0}_{i}(\vec x) \wedge \nu_i^{\apr_1}(u,\vec y)) $.
\end{itemize}
So, we have:
{
\begin{eqnarray*}
{\sf LIN}^{\sf split} + {\sf C}(\vec x;\vec y) \vdash \exists u\, \alpha_0(u,\vec x,\vec y) & \iff & 
\exists u\, (\apr_0(u) \wedge \alpha_0(u,\vec x,\vec y)) \;\vee \\
&& \exists u\, (\apr_1(u) \wedge \alpha_0(u,\vec x,\vec y)) \\
& \iff & \exists u\, (\apr_0(u) \wedge \bigvee_{i<n}(\eta^{\apr_0}_{i}(u,\vec x) \wedge \theta_i^{\apr_1}(\vec y))) \vee \\
&& \exists u\, (\apr_1(u) \wedge\bigvee_{j<m}(\kappa^{\apr_0}_{i}(\vec x) \wedge \nu_i^{\apr_1}(u,\vec y))) \\
& \iff &  \bigvee_{i<n}( \exists u\, (\apr_0(u) \wedge \eta^{\apr_0}_{i}(u,\vec x)) \wedge \theta_i^{\apr_1}(\vec y)) \;\vee \\
&& \bigvee_{j<m}(\kappa^{\apr_0}_{i}(\vec x) \wedge  \exists u\, (\apr_1(u) \wedge \nu_i^{\apr_1}(u,\vec y))) \\
& \iff &  \bigvee_{i<n}(   (\exists u\,\eta_{i}(u,\vec x))^{\apr_0} \wedge \theta_i^{\apr_1}(\vec y)) \;\vee \\
&& \bigvee_{j<m}(\kappa^{\apr_0}_{i}(\vec x) \wedge (\exists u\, \nu_i (u,\vec y))^{\apr_1}) 
\end{eqnarray*}
}
So, we are done. (We note that the calculation also works when one of the domains is empty.)
\end{proof}

\noindent
Let $\alpha$ be a sentence of signature $\Lambda$ with ${\sf LIN} \dashv \alpha$. 
Let $\eta_i$, for $i<n$ and $\theta_j$, for $j < \ell$, be the sentences produced for $\alpha$ in Theorem~\ref{nerdsmurf}.
 Let $s$ be a $0,1$-sequence of of length $\ell$. We write $y \in (x,\infty)$ for $x<y$.
We define:
\begin{itemize}
\item
$\beta^s(x)$ is the conjunction of the sentences $\theta_j^{(x,\infty)}$ when $s_j=1$ and \\ $\neg\, \theta^{(x,\infty)}_j$ if $s_j=0$.  
\end{itemize} 

\noindent
We will say that \emph{$x$ witnesses $s$} for $\beta_s(x)$. We note that each $x$ witnesses a unique $s$.
We define the theory $\valf{\alpha}$ as follows.

\begin{enumerate}[$\valf{\alpha}$1.]
\item 
$\alpha$
\item 
$\exists x\, \forall y\, x\leq y$ (Zero)
\item
$\forall x\, \forall y\, (x< y \to \exists z\, (x <z \wedge \forall u\, (x < u \to z \leq u)))$ (Restricted Successor)
\item
$\exists x \, \beta^s(x) \to \exists x\, (\beta^s(x) \wedge \forall y\, (x< y \to \neg\, \beta^s(y)))$, for each $s:\ell \to 2$.\\
(If $s$ has a witness at all, it has a largest witness.)
\end{enumerate}

\begin{theorem}\label{koksmurf}
Let $\alpha$ be a sentence in the language of linear orderings expanded with finitely many unary
predicate symbols.
Suppose ${\sf LIN} \dashv \alpha$. Then,
$\alpha$ has a finite model iff  $\valf{\alpha}$ is consistent.
\end{theorem}

\begin{proof}
The left-to-right direction is immediate. 

\medent
We prove right-to-left.
Suppose $\valf{\alpha}$ is consistent. Then, $\valf{\alpha}$ has a model $\mathcal N$.
We write $\mathfrak n(\mathcal N)$ for the number of $s$ that are witnessed infinitely
often in $\mathcal N$.
Let $\mathcal M$ be a model of $\valf{\alpha}$  such that $\mathfrak n(\mathcal M)$ is minimal. Suppose $\mathcal M$ 
is infinite.

\medent
The model $\mathcal M$ begins with a copy of $\omega$. By the pigeon-hole principle, there is an $s^\ast$ that is witnessed infinitely often in
this copy of $\omega$. 
Clearly, the initial copy of $\omega$ contains a smallest witness $a$ of $s^\ast$. Let $b$ be the maximal witness of $s^\ast$.

We now remove the interval $\opcl{a,b}$ from $\mathcal M$, thus obtaining a new model $\mathcal M'$. 

\medent
We claim that $\mathcal M'$ again satisfies $\valf{\alpha}$.

To prove ($\valf{\alpha}1$), we note that:
$\mathcal M \models \eta_i^{[0,a]}$ iff $\mathcal M' \models \eta_i^{[0,a]}$. Moreover, we have
$\mathcal M \models \theta_j^{(a,\infty)}$ iff $\mathcal M \models \theta_j^{(b,\infty)}$, and 
$\mathcal M \models \theta_j^{(b,\infty)}$ iff $\mathcal M' \models \theta_j^{(a,\infty)}$.
So, $\mathcal M \models \theta_j^{(a,\infty)}$ iff $\mathcal M' \models \theta_j^{(a,\infty)}$.
Hence, $\mathcal M' \models \alpha$. 

The preservation of ($\valf{\alpha}2$) and ($\valf{\alpha}3$) is immediate.

Finally, consider any $s$ that is witnessed in $\mathcal M'$.
In case $s$ only has $\mathcal M'$-witnesses in $[0,a]$ we are done, since $[0,a]$ is finite. In case $s$ has an $\mathcal M'$-witness in
$(a,\infty)$, then it has an $\mathcal M$-witness in
$(a,\infty)$, since the question whether $c$ is a witness only depends on what happens above $c$.
It follows that $s$ has a maximal $\mathcal M$-witness in $(a,\infty)$, and, hence, $s$ has a maximal
$\mathcal M'$-witness in $(a,\infty)$. This gives us ($\valf{\alpha}4$).

\medent
We note that if $s$ is witnessed infinitely often in $\mathcal M'$, then it  witnessed infinitely often in $\mathcal M$.
On the other hand, $s^\ast$ is not witnessed in $(b,\infty)$, so it is only witnessed finitely often
in $\mathcal M'$. Thus, $\mathfrak n(\mathcal M') < \mathfrak n(\mathcal M)$. This contradicts the minimality of
 $\mathfrak n(\mathcal M)$. 
 
 We may conclude that $\mathcal M$ must be finite. 
\end{proof}

\noindent 
Since, Theorem~\ref{koksmurf} tells us that the property of $\alpha$ having a finite model is both recursively enumerable and
co-recursively enumerable, we have the following corollary.

\begin{corollary}\label{beslissmurf}
Suppose $\alpha$ is a sentence in the language of {\sf LIN} expanded with finitely many
unary predicate symbols and ${\sf LIN} \dashv\alpha$. Then, it is decidable whether
$\alpha$ has a finite model.
\end{corollary}

\begin{question}\label{Q6}
Suppose $\alpha$ is a sentence in the language of {\sf LIN} expanded with finitely many
unary predicate symbols and ${\sf LIN} \dashv\alpha$.
Is there a better algorithm than the one suggested for Corollary~\ref{beslissmurf} to determine whether $\alpha$ has a finite model?
\end{question}

\noindent 
 We write:
 \begin{itemize}
 \item
  $\gamma\tupel{\delta}\theta$
iff $(\gamma \wedge \delta) \vee (\theta\wedge \neg\, \delta)$.
 \end{itemize}
 Thus, $\gamma\tupel{\delta}\theta$ means: $\gamma$ if $\delta$, else $\theta$.  

We are now ready and set to prove Theorem~\ref{nerdsmurf}(b).

\begin{proof}[Proof of Theorem~\ref{nerdsmurf}\textup (b\textup)]
Let ${\sf LIN} \dashv A$, $U := A + {\sf INF}$ and $\alpha^\star := (A+{\sf UB}(P))$.
Suppose $\alpha \vdash U$ and, suppose, to get a contradiction, that
(\ddag) $\alpha \not \jump_{\sf pf} \alpha^\star$. We note that this is equivalent to $\alpha \not \jump_{\sf pf} {\sf UB}(P)$, since
$A$ is in the $<$-language.

\medent
We show that the theory $W :=\alpha + \verz{\neg\, {\sf UB}(\phi) \mid \phi \in {\sf Form}^1_\Lambda}$ is consistent.
Here ${\sf Form}^1_\Lambda$ is the set of $\Lambda$-formulas with at most the variable $v$ free.
Suppose $W$ were inconsistent. Then, by compactness, we would have
$\alpha\vdash \bigvee_{i<k}{\sf UB}(\phi_i)$, for some $k$ and for some choice of the $\phi_i$.

Let $\psi := \phi_0 \tupel{{\sf UB}(\phi_0)}(\phi_1\tupel{{\sf UB}(\phi_1)}(\ldots))$. 
Clearly, it would follow that $\alpha \vdash {\sf UB}(\psi)$ and, hence, $\alpha \jump_{\sf pf} {\sf UB}(P)$. \emph{Quod non}, by Assumption~(\ddag).

\medent
Let $\mathcal M$ be a model of $W$. In $W$ every definable non-empty set has both a maximum and a minimum. 
We verify that $\mathcal M$ satisfies $\valf{\alpha}$. Clearly $\mathcal M$ satisfies ($\valf{\alpha}1$) and ($\valf{\alpha}2$).
Suppose $a$ is not maximal and $a$ has no direct successor.
Then, there is a minimal such element $a^\star$. Since $a^\star$ is definable, there is a least $b > a^\star$. But this $b$ must be the direct successor
of $a^\star$. A contradiction. So, every non-maximal $a$ has a direct successor.  This gives us ($\valf{\alpha}3$).
Finally, ($\valf{\alpha}4$) is again immediate.

We have shown that $\valf{\alpha}$ is consistent and, hence, by Theorem~\ref{koksmurf}, 
has a finite model. A contradiction. We may conclude that Assumption~(\ddag) is false.
\end{proof}

\bibliographystyle{alpha}
\bibliography{provint}

\begin{thebibliography}{TMR53}

\bibitem[Bar17]{barw:admi17}
J.~Barwise.
\newblock {\em Admissible sets and structures}, volume~7 of {\em Perspectives
  in Mathematical Logic}.
\newblock Cambridge University Press, 2017.

\bibitem[Bus86]{buss:boun86}
S.R. Buss.
\newblock {\em Bounded Arithmetic}.
\newblock Bibliopolis, Napoli, 1986.

\bibitem[CV58]{crai:fini58}
W.~Craig and R.L. Vaught.
\newblock Finite axiomatizability using additional predicates.
\newblock {\em The Journal of Symbolic Logic}, 23(3):289--308, 1958.

\bibitem[Fri07]{frie:inte07}
H.~Friedman.
\newblock Interpretations according to {T}arski.
\newblock This is one of the 2007 Tarski Lectures at Berkeley. The lecture is
  available at \\ {\tt http://www.math.osu.edu/\textasciitilde
  friedman.8/pdf/Tarski1,052407.pdf}, 2007.

\bibitem[Fuj10]{fuji:rela10}
K.~Fujimoto.
\newblock Relative truth definability of axiomatic truth theories.
\newblock {\em Bulletin of Symbolic Logic}, 16(3):305--344, 2010.

\bibitem[HP93]{haje:meta91}
P.~H{\'a}jek and P.~Pudl{\'a}k.
\newblock {\em Metamathematics of First-Order Arithmetic}.
\newblock Perspectives in Mathematical Logic. Springer, Berlin, 1993.

\bibitem[Imm12]{imme:desc12}
N.~Immerman.
\newblock {\em Descriptive complexity}.
\newblock Springer Science {\&} Business Media, 2012.

\bibitem[Kle52]{klee:fini52}
S.C. Kleene.
\newblock Finite axiomatizability of theories in the predicate calculus using
  additional predicate symbols.
\newblock In {\em Two Papers on the Predicate Calculus}, number~10 in Memoirs
  of the American Mathematical Society, pages 27--68. American Mathematical
  Society, Providence, 1952.

\bibitem[MPS90]{myci:latt90}
J.~Mycielski, P.~Pudl{\'a}k, and A.S. Stern.
\newblock {\em A lattice of chapters of mathematics (interpretations between
  theorems)}, volume~84 of {\em Memoirs of the American Mathematical Society}.
\newblock AMS, Providence, Rhode Island, 1990.

\bibitem[Tar54]{tars:gene54}
A~Tarski.
\newblock A general theorem concerning the reduction of primitive notions.
\newblock {\em The Journal of Symbolic Logic}, 19:158--159, 1954.

\bibitem[TMR53]{tars:unde53}
A.~Tarski, A.~Mostowski, and R.M. Robinson.
\newblock {\em Undecidable theories}.
\newblock North--Holland, Amsterdam, 1953.

\bibitem[Vis13]{viss:what13}
A.~Visser.
\newblock What is sequentiality?
\newblock In P.~C{\'{e}}gielski, Ch. Cornaros, and C.~Dimitracopoulos, editors,
  {\em New Studies in Weak Arithmetics}, volume 211 of {\em CSLI Lecture
  Notes}, pages 229--269. CSLI Publications and Presses Universitaires du
  P{\^{o}}le de Recherche et d'Enseingement Sup{\'{e}}rieur Paris-est,
  Stanford, 2013.

\bibitem[Vis17]{viss:onq17}
A.~Visser.
\newblock On {{\sf Q}}.
\newblock {\em Soft Computing}, 21(1):39--56, 2017.

\bibitem[VisXX]{viss:inteXX}
A.~Visser.
\newblock The interpretation existence lemma.
\newblock In {\em Feferman on Foundations}, Outstanding Contributions to Logic.
  Springer, 20XX.
\newblock To appear.

\end{thebibliography}

\appendix

\section{Proof of Nine Cases}\label{ninecases}
We formulate the generalized version of our main theorem.

\begin{theorem}\label{grotesmurf}
Consider any recursively enumerable theory $U$ that is not finitely axiomatizable and any finite expansion $\Theta$ of the signature of $U$ with ${\sf M}_\Theta \geq 2$. 
Let $\sqsubseteq$ be one of $\dashv$,  $\dashv_{\sf c}$, $\dashv_{\sf mc}$ and
let $\preceq$ be one of $\dashv$, $\jumpb_{\sf pf}$ or $\jumpb$. 

Then, for all $\alpha \sqsupseteq U$, there
is a $\beta$ with $\alpha \succ \beta \sqsupseteq U$.
\end{theorem}

\noindent
We note that Theorem~\ref{grotesmurf} tells us that if $U$ is not finitely axiomatizable, then there is no $\preceq$-has no minimal element, 
if $\sqsubseteq$ is one of $\dashv$, $\dashv_{\sf c}$, $\dashv_{\sf mc}$ and if
$\preceq$ is one of $\dashv$, $\jumpb_{\sf pf}$ or $\jumpb$.  Since the $\alpha$ are closed under infima, an element is minimal iff it is a minimum.
Thus, Theorem~\ref{grotesmurf} tells us that, if $U$ is not finitely axiomatizable, then the $\alpha$ such that $U\sqsubseteq\alpha$ have
no minimum. 

The following two lemmas allow us to derive the nine cases of Theorem~\ref{grotesmurf} from Theorem~\ref{heiligesmurf}.

\begin{lemma}\label{babysmurf}
Suppose:
\begin{enumerate}[i.]
\item
For all $\alpha$, we have: if $U \sqsubseteq_0 \alpha$, then $U \dashv \alpha$.
\item
For all $\alpha$, we have: if $U \sqsubseteq_1 \alpha$, then $U \sqsubseteq_0 \alpha$.
\item
For all $\alpha,\beta$, we have: if $U \sqsubseteq_1 \alpha$ and $U \dashv \beta \preceq \alpha$, then
$U \sqsubseteq_1 \beta$. 
 \end{enumerate}
Suppose \textup(a\textup):
 for all $\alpha$ such that $U \sqsubseteq_0 \alpha$, there is a $\beta$ with $U\sqsubseteq_0 \beta \prec \alpha$.
Then, we have \textup(b\textup):
 for all $\alpha$ such that $U \sqsubseteq_1 \alpha$, there is a $\beta$ with $U\sqsubseteq_1 \beta \prec \alpha$.
\end{lemma}

\begin{proof}
Suppose we have (a) and $U \sqsubseteq_1 \alpha$. Then, by (ii), $U \sqsubseteq_0\alpha$.
Thus, by (a), there is a $\beta$ with $U\sqsubseteq_0 \beta \prec \alpha$.  By (i), we find $U\dashv \beta \prec \alpha$.
Since, $U \sqsubseteq_1 \alpha$, we may conclude, by (iii), that $U \sqsubseteq_1 \beta \prec \alpha$.
\end{proof}

\begin{lemma}\label{kleutersmurf}
Suppose:
\begin{enumerate}[i.]
\item
For all $\alpha$ and $\beta$, we have: if  $\beta \dashv\alpha$, then $\beta \preceq_1 \alpha$.
\item
For all $\alpha$ and $\beta$, we have: if  $\beta \preceq_1\alpha$, then $\beta \preceq_0 \alpha$.
\item
$\preceq_1$ is transitive.
\item
For all $\alpha$, if $U \sqsubseteq \alpha$, then $U \dashv\alpha$.
\item
For all $\alpha$ and $\beta$, if $U \sqsubseteq \alpha$ and $U \dashv \beta \dashv \alpha$, then $U \sqsubseteq \beta$.
\end{enumerate}
Suppose  \textup(a\textup):
 for all $\alpha$ such that $U \sqsubseteq \alpha$, there is a $\beta$ with $U\sqsubseteq \beta \prec_0 \alpha$.
 Then, we have \textup(b\textup):
 for all $\alpha$ such that $U \sqsubseteq \alpha$, there is a $\beta$ with $U\sqsubseteq \beta \prec_1 \alpha$.
\end{lemma}

\begin{proof}
Suppose we have (a) and $U \sqsubseteq \alpha$. Then, by (a), there is $\gamma$ with $U\sqsubseteq \gamma \prec_0 \alpha$.
We find $\beta := (\alpha \vee \gamma) \dashv \alpha$. Hence, by (i), $\beta \preceq_1 \alpha$. 

Suppose we would have
$\alpha \preceq_1 \beta$. We have, $\beta \dashv \gamma$, and, hence, by (i), $\beta \preceq_1 \gamma$. So, by (iii), we would have
$\alpha \preceq_1 \gamma$. But then, by (ii), we would have $\alpha \preceq_0\gamma$. Quod non. We may conclude that $\beta \prec \alpha$.
 
 Finally, we have $U \sqsubseteq \alpha$ and $U \sqsubseteq \gamma$. Hence, by (iv), $U \dashv \alpha$ and $U \dashv \gamma$. It follows that
 $U \dashv \beta \dashv \alpha$. But then, by (v), we have $U \sqsubseteq \beta$.
\end{proof}

\begin{proof}{Proof of Theorem~\ref{grotesmurf}}
We note that, if we interpret $\preceq$ as any of $\dashv$, $\jumpb_{\sf pf}$ or $\jumpb$, then any of the
pairs $\dashv, \dashv_{\sf c}$ and $\dashv_{\sf c}, \dashv_{\sf mc}$ satisfies the assumptions (i)-(iii) of Lemma~\ref{babysmurf}.
By applying Lemma~\ref{babysmurf}, we find that in order to prove Theorem~\ref{grotesmurf}, it suffices to prove the cases
where $\sqsubseteq$ is $\dashv$.

We also note that, if we interpret $\sqsubseteq$ as any of $\dashv$, $\dashv_{\sf c}$ or $\dashv_{\sf mc}$, then any
of the pairs $\jumpb_{\sf pf},\dashv$ and $\jumpb, \jumpb_{\sf pf}$ satisfies the assumptions (i)-(v) of Lemma~\ref{kleutersmurf}.
By applying Lemma~\ref{kleutersmurf}, we find that in order to prove Theorem~\ref{grotesmurf}, it suffices to prove the cases
where $\preceq$ is $\jumpb$.

Thus, we may conclude that we have Theorem~\ref{grotesmurf}.
\end{proof}

\section{Some Results in the Environment of Our Problem}\label{environment}
The following result is Theorem~5.3 of \cite{viss:inteXX}.

\begin{theorem}\label{exosmurf}
Let $A$ and $U$ be theories, where $A$ is finitely axiomatized and $U$ is recursively enumerable and sequential. 
Suppose $A \rhdnneq U$. Then, there is a finitely axiomatized theory $B$ such that $A \rhdnneq B \rhdnneq U$.
Moreover, if $A$ is sequential,  $B$ is sequential too.
\end{theorem}

\begin{question}\label{Q7}
Can we extend Theorem~\ref{exosmurf} to a wider class of theories $U$?
\end{question}

\noindent The following result is Theorem~2 of \cite{viss:onq17}. The theory {\sf R}
is the Tarski-Mostowski-Robinson theory {\sf R} from \cite{tars:unde53}.

\begin{theorem}
Suppose ${\sf R} \subseteq A$, where $A$ is finitely axiomatized and consistent.
Then, there is a finitely axiomatized $B$ such that ${\sf R} \subseteq B \subseteq A$ and
$B \nrhd A$.
\end{theorem}

\begin{example}
It is very well possible that a non-finitely axiomatizable theory has a 
minimal finite extension in the same language w.r.t. $\dashv$. 
An example is Peano Arithmetic that has the inconsistent theory as its only finite extension in the same language. 

If the reader objects
to having the inconsistent theory as an example, let e.g. $A$ be the conjunction of the axioms of {\sf EA} plus $\opr_{\sf PA}\bot$.
Let $U$ be the theory axiomatized by axioms $B \vee A$, where $B$ is an axiom of {\sf PA}.
Clearly $A$ is a finite consistent extension of $U$. Suppose $C$ is another such extension. We note that $C\wedge \neg \, A$ extends {\sf PA},
so $C\wedge \neg\, A \vdash \bot$ and, hence, $C \vdash A$.
\end{example}

\section{List of Questions}

\begin{enumerate}[Q1.]
\item
It there a recursively enumerable theory $U$ with an NP set of finite models (modulo isomorphism)
such that there is no $\alpha$ with $U \dashv_{\sf c} \alpha$? (This is Question~\ref{Q1}.)
\item
Can we find a recursively enumerable $U$ and an $\alpha$ in an expanded language, such that $U \dashv_{\sf c} \alpha$,
where there is no $\beta$ such that $U \dashv_{\sf mc} \beta$? (This is Question~\ref{Q2}.)
\item
Consider any theory $U$ that is not finitely axiomatizable and any finite expansion $\Theta$ of the signature of $U$ with ${\sf M}_\Theta \geq 2$. 
Suppose $\alpha$ is not interpretable (in the full sense of interpretability) in $U$. Is there an extension $\beta$ of $U$, such that $\alpha$ is not interpretable in $\beta$?
(This is Question~\ref{Q3}.)
\item
Consider any theory $U$ that is not finitely axiomatizable and any finite expansion $\Theta$ of the signature of $U$ with ${\sf M}_\Theta \geq 2$. 
We take as the length of a proof the number of symbols in the proof written in a fixed finite alphabet. 
We define $\beta \preceq_{\sf sp} \alpha$ iff there is a polynomial $P(x)$ such that, for every $A$ (of the language of $U$),
 if $A$ is provable from $\beta$ by a proof of the length $n$, then $A$ is provable from $\alpha$ by a proof of  length $\leq P(n)$.
Can there be a $ \preceq_{\sf sp}$-minimal conservative extension $\alpha^\star$ of $U$? (This is Question~\ref{Q4}.)
\item
Is there an example of a theory $U$ and an $\alpha$ in the language of $U$ extended with a non-empty finite signature of unary predicate
symbols, such that $\alpha$ is a $\dashv$-minimal conservative extension of $U$? (This is Question~\ref{Q5}.)
\item
Suppose $\alpha$ is a sentence in the language of {\sf LIN} expanded with finitely many
unary predicate symbols and ${\sf LIN} \dashv\alpha$.
Is there a better algorithm than the one suggested for Corollary~\ref{beslissmurf} to determine whether $\alpha$ has a finite model? (This is Question~\ref{Q6}.)
\item
Can we extend Theorem~\ref{exosmurf} to a wider class of theories $U$? (This is Question~\ref{Q7}.)
\end{enumerate}

\end{document}